\documentclass[12pt,a4paper]{amsart}

\usepackage[latin1]{inputenc}
\usepackage{enumerate}
\usepackage{rotating}
\usepackage{fancybox}
\usepackage{float}
\restylefloat{table}
\usepackage{graphicx}

\usepackage{amsmath}

\usepackage{amssymb}

\usepackage{amsthm}

\usepackage[arrow, matrix, curve]{xy}

\usepackage{a4wide}

\usepackage{txfonts}

\usepackage{hyperref}

\usepackage{color}

\theoremstyle{plain}
\newtheorem{thma}{Main Theorem}
\newtheorem{thm}{Theorem}[section]
\newtheorem{prop}[thm]{Proposition}
\newtheorem{cor}[thm]{Corollary}

\theoremstyle{definition}
\newtheorem{defn}[thm]{Definition}
\newtheorem{lem}[thm]{Lemma}

\newtheorem{ex}[thm]{Example}

\numberwithin{equation}{thm}

\newcommand{\emphbf}[1]{\emph{\textbf{#1}}}

\DeclareMathAlphabet{\mathpzc}{OT1}{pzc}{m}{it}

\DeclareMathOperator{\Kopf}{top}
\DeclareMathOperator{\rad}{rad}

\DeclareMathOperator{\cx}{cx}
\DeclareMathOperator{\Dist}{Dist}
\DeclareMathOperator{\SL}{SL}
\DeclareMathOperator{\Ad}{Ad}
\DeclareMathOperator{\image}{Im}
\DeclareMathOperator{\charac}{char}
\DeclareMathOperator{\St}{St}
\DeclareMathOperator{\modu}{mod}
\DeclareMathOperator{\Hom}{Hom}
\DeclareMathOperator{\Ext}{Ext}

\DeclareMathOperator{\Lie}{Lie}
\DeclareMathOperator{\heart}{ht}
\DeclareMathOperator{\soc}{soc}

\DeclareMathOperator{\PGL}{PGL}
\DeclareMathOperator{\ad}{ad}

\DeclareMathOperator{\gr}{gr}
\DeclareMathOperator{\ch}{ch}

\begin{document}

\title{Auslander-Reiten theory of Frobenius-Lusztig kernels}
\author{Julian K\"ulshammer}
\address{Christian-Albrechts-Universit\"at zu Kiel, Ludewig-Meyn-Str. 4, 24098 Kiel, Germany}
\email{kuelshammer@math.uni-kiel.de}
\thanks{Supported by the D.F.G. priority program SPP1388 "Darstellungstheorie"}

\begin{abstract}
In this paper we show that the tree class of a component of the stable Auslander-Reiten quiver of a Frobenius-Lusztig kernel is one of the three infinite Dynkin diagrams. For the special case of the small quantum group we show that the periodic components are homogeneous tubes and that the non-periodic components have shape $\mathbb{Z}[\mathbb{A}_\infty]$ if the component contains a module for the infinite-dimensional quantum group.
\end{abstract}

\maketitle

\section*{Introduction}

In the representation theory of algebras, two cases have served as paradigms, that are group algebras and hereditary algebras. It was proven by Ringel \cite{Rin78} and Erdmann \cite{Er95} that in both cases only one tree class can occur for a component of the stable Auslander-Reiten quiver in the case of wild representation type. Building on these results the Auslander-Reiten theory of other classes of wild algebras has been studied (see e.g. \cite{LdlP97}, \cite{Er96}, \cite{BE11}).\\

In 1990, Lusztig \cite{L90} defined a finite dimensional quantum group, the small quantum group, which has several similarities with restricted enveloping algebras although defined in characteristic zero. In 2009, Drupieski \cite{DruPhD} gave a generalization to positive characteristic, that gives also quantum analogues of higher Frobenius kernels, which he called (higher) Frobenius-Lusztig kernels. Building on results by Feldvoss and Witherspoon in \cite{FW09} in a recent paper \cite{K11}, the author has shown that the non-simple blocks of a Frobenius-Lusztig kernel are of wild representation type in all but two cases. Therefore the class of Frobenius-Lusztig kernels can also serve as an example for understanding the Auslander-Reiten theory of wild self-injective algebras.\\

For self-injective algebras where the modules have finite complexity (which was proven for the Frobenius-Lusztig kernels in \cite{Dru11}), Kerner and Zacharia provided a version of Webb's Theorem, that limits the tree classes of components to Euclidean and infinite Dynkin diagrams. For Frobenius-Lusztig kernels we are able to boil them down and prove

\begin{thma}
Let $\mathfrak{g}\neq \mathfrak{sl}_2$ be a finite dimensional complex simple Lie algebra. Then all non-periodic components of the Frobenius-Lusztig kernel associated to $\mathfrak{g}$ are of the form $\mathbb{Z}[\mathbb{A}_\infty]$, $\mathbb{Z}[\mathbb{D}_\infty]$ or $\mathbb{Z}[\mathbb{A}_\infty^\infty]$. In the case of the small quantum group, if the component additionally contains the restriction of a module for the (infinite dimensional) Lusztig form of the quantum group, then it can only be $\mathbb{Z}[\mathbb{A}_\infty]$.
\end{thma}

For the special case where the tree class is the Kronecker quiver we can even give a full classification of all algebras having such a component. As a corollary we obtain that all such algebras are special biserial.\\

In this article, $k$ denotes an algebraically closed field. All vector spaces will be assumed to be finite dimensional unless otherwise stated. If $G$ is a semisimple algebraic group, the corresponding Lie algebra will be denoted by $\mathfrak{g}$, the set of roots will be denoted by $\Phi$, the set of simple roots is denoted $\Pi$, the corresponding set of positive roots is denoted by $\Phi^+$, $h$ is the Coxeter number. For general theory we refer the reader to \cite{Jan03}.\\

For a general introduction to Auslander-Reiten theory we refer the reader to \cite{ARS95}, \cite{ASS06} or \cite{B95}. We denote the syzygy functor by $\Omega$ and the Auslander-Reiten translation by $\tau$.\\

Our paper is organized as follows: In Section 1 we recall the basic definitions and results in the theory of support varieties that are needed in the remainder. Section 2 is devoted to a recapitulation of the definition of Frobenius-Lusztig kernels and their properties. Section 3 recalls Webb's Theorem for Frobenius-Lusztig kernels and describes the periodic components in more detail. Section 4 classifies all algebras having a component of Kronecker tree class in terms of quivers and relations. In Section 5 we prove statements on graded modules over quantum groups and Frobenius extensions that we need in the proof of our main result. Section 6 contains the statement and proof of the main result. In Section 7 we study the case that the tree class is $\mathbb{A}_\infty$ in more detail and prove statements about simple modules in such components.

\section{Support varieties for (fg)-Hopf algebras}

The notion of a support variety was first defined for group algebras and later generalized to various classes of algebras, whose cohomology satisfies certain finiteness conditions. In our case the following one is satisfied:\\

A finite dimensional Hopf algebra $A$ is said to satisfy (fg), if the even cohomology ring $H^{ev}(A,k)$ is finitely generated and the $H^{ev}(A,k)$-modules $\Ext^\bullet(M,N)$ are finitely generated for all finite dimensional $A$-modules $M$ and $N$.\\

If $A$ is an (fg)-Hopf algebra, then the support variety of a module $M$ is defined as the variety $\mathcal{V}(M)$ associated to the ideal $\ker \Phi_M$, where $\Phi_M:H^{ev}(A,k)\to \Ext^{\bullet}_A(M,M)$ is induced by $-\otimes M$. In the remainder of this section we will list some properties of these varieties that we will use in this paper. The proofs for the group algebra case (see e.g. \cite{B91}) generalize without much difficulty. Some of them may be found in \cite{Bro98} or \cite{FW09}.

\begin{prop}\label{generalsupportfacts}
Let $A$ be an (fg)-Hopf algebra. Then:
\begin{enumerate}[(i)]
\item $\dim \mathcal{V}(M)=\cx M=\gamma(\Ext^\bullet_A(M,M))$, where $\cx M$, the complexity of $M$, is the rate of growth $\gamma$ of a minimal projective resolution of $M$. For a sequence of vector spaces $V_n$, the rate of growth $\gamma(V_*)$ is the smallest number $d$, such that there exists a constant $C>0$ with $\dim V_n\leq Cn^{d-1}$.
\item $M$ is projective iff $\cx M=0$.
\item $M$ is $\Omega$-periodic iff $\cx M=1$. Furthermore if the even cohomology ring is generated in degree $n$, then the $\Omega$-period of every periodic module is a divisor of $n$.
\item If $M$ is indecomposable then $\mathbb{P}\mathcal{V}(M)$, the projectivized support variety for $M$, is connected.
\item $\mathcal{V}(M)=\mathcal{V}(N)$ if $M$ and $N$ belong to the same connected component of the stable Auslander-Reiten quiver.
\end{enumerate}
\end{prop}

The following statement is a generalization of \cite[Lemma 2.1]{Fa95} and \cite[Theorem 1.1]{Fa11}:

\begin{thm}\label{upperbound}
Let $A$ be an (fg)-Hopf algebra. Let $M$ be a module, such that the commutative graded subalgebra $S:=\image \Phi_M\subseteq \Ext^\bullet_A(M,M)$ is generated by $\bigoplus_{b|a} S_b$ for some $a\in \mathbb{N}$. Then $\cx M\leq \dim \Ext^{an}_A(M,M)$ for all $n\geq 1$.
\end{thm}

\begin{proof}
Denote by $T_{(n)}$ the subalgebra of $S$ generated by the subspace $S_{an}$ of homogeneous elements of degree $an$. Since $S$ is generated by $\bigoplus_{b|a} S_b$ as an algebra, it follows that it is generated as a $T_{(n)}$-algebra by finitely many integral elements. By \cite[Corollary 4.5]{Eis95} it is therefore finitely generated as a $T_{(n)}$-module. Since $\Ext^\bullet_A(M,M)$ is finitely generated as an $S$-module, we also have that it is finitely generated as a $T_{(n)}$-module. Hence $\cx M=\gamma(\Ext^\bullet_A(M,M))=\gamma(T_{(n)})\leq \dim S_{an}\leq \dim \Ext^{an}_A(M,M)$, where the inequalities follow from the foregoing proposition, as $\Ext^\bullet(M,M)$ is a finitely generated module over $T_{(n)}$, the fact that equality holds for the polynomial ring in finitely many variables, and the fact that $S_{an}$ is a subspace of $\Ext^{an}_A(M,M)$, respectively.
\end{proof}

\section{Frobenius-Lusztig kernels}

In this section we recall some basic properties of Frobenius-Lusztig kernels that we use throughout this paper.\\

Let $\ell\geq 3$ be an odd integer (We also suppose that $3$ does not divide $\ell$ if the root system mentioned in the statements contains a component of type $\mathbb{G}_2$). We fix a primitive $\ell$-th root of unity $\zeta$ and denote by $U_\zeta(\mathfrak{g})$ the Lusztig form of the quantum group at a root of unity $\zeta$ (cf. \cite{DruPhD}). For $r\in \mathbb{N}_0$ the $r$-th Frobenius-Lusztig kernel $U_\zeta(G_r)$ is defined as the subalgebra of $U_\zeta(\mathfrak{g})$ generated by the set $\{E_\alpha, E_\alpha^{(p^i\ell)}, F_\alpha, F_\alpha^{(p^i\ell)}, K_\alpha^{\pm 1}|\alpha\in \Pi, 0\leq i<r\}$, where $E_\alpha^{(n)}$ and $F_\alpha^{(n)}$ denote the $n$-th divided powers (see e.g. \cite{Jan96}). Note that in the case of characteristic zero only the zeroth Frobenius-Lusztig kernel exists. The zeroth Frobenius-Lusztig kernel is also called small quantum group.\\

For a simple algebraic group scheme $G$, denote its $r$-th Frobenius kernel by $G_r$ and the corresponding algebra of distributions, i.e. the dual of the coordinate ring of $G_r$, by $\Dist(G_r)$ (see e.g. \cite{Jan03}). Then $U_\zeta(G_0)$ is a normal Hopf subalgebra of $U_\zeta(G_r)$ with Hopf quotient isomorphic to $\Dist(G_r)$. Recall that all simple modules for $U_\zeta(G_0)$ are also modules for $U_\zeta(\mathfrak{g})$ and hence for all $U_\zeta(G_r)$ and that for every $r$ there exists exactly one simple projective $U_\zeta(G_r)$-module denoted by $\St_{p^r\ell}$. Furthermore each simple $U_\zeta(G_r)$-module is a tensor product of a $\Dist(G_r)$-module (viewed as a module for $U_\zeta(G_r)$) and a $U_\zeta(G_0)$-module. In \cite{K11} we proved that there is another embedding $F$ of $\modu \Dist(G_r)$ into $\modu U_\zeta(G_r)$, induced by tensoring with $\St_\ell$, whose image is a direct sum of (categorical) blocks.\\

The main result of \cite{K11} was the following:

\begin{thm}
Let $\mathfrak{g}$ be a finite dimensional complex simple Lie algebra. Let $\ell>1$ be an odd integer not divisible by $3$ if $\Phi$ is of type $\mathbb{G}_2$.
\begin{enumerate}[(i)]
\item The only representation-finite block of $U_\zeta(G_r)$ is the semisimple block corresponding to $\St_{p^r\ell}$.
\item Furthermore for $\charac k=0$ assume that $\ell$ is good for $\Phi$ (i.e. $\ell \geq 3$ for type $\mathbb{B}_n$, $\mathbb{C}_n$ and $\mathbb{D}_n$, $\ell\geq 5$ for type $\mathbb{E}_6$, $\mathbb{E}_7$ and $\mathbb{G}_2$ and $\ell\geq 7$ for $\mathbb{E}_8$) and that $\ell>3$ if $\Phi$ is of type $\mathbb{B}_n$ or $\mathbb{C}_n$. For $\charac k=p>0$ assume that $p$ is good for $\Phi$ and that $\ell>h$. If $U_\zeta(G_r)$ satisfies (fg), then $U_\zeta(G_r)$ is tame iff $G=\SL_2$ and $r=0$ or $r=1$ and the block lies in the image of $F$.
\end{enumerate}
\end{thm}

\section{Webb's Theorem}

In \cite{KZ11} Kerner and Zacharia have shown that for self-injective algebras, components of finite complexity have a  very special shape. In our context their theorem reads as follows:

\begin{thm}
If $\Theta$ is a non-periodic component (i.e. there does not exist $M\in \Theta$ with $\tau^mM\cong M$ for some $m>0$ or equivalently $\Omega^mM\cong M$ since the algebra is symmetric and hence $\tau\cong \Omega^2$) of the stable Auslander-Reiten quiver of $U_\zeta(G_r)$, then $\Theta\cong \mathbb{Z}[\Delta]$, where $\Delta$ is a Euclidean or infinite Dynkin diagram (for the Euclidean diagram $\tilde{\mathbb{A}}_n$ one has to take a non-cyclic orientation when constructing $\mathbb{Z}[\tilde{\mathbb{A}}_n]$).
\end{thm}

\begin{proof}
In \cite[Theorem 6.3.1]{Dru11} it is shown that all modules have finite complexity. Hence the result is a direct consequence of \cite[Main Theorem]{KZ11}.
\end{proof}

The question of periodic components was settled much earlier by Happel, Preiser and Ringel in \cite{HPR80}. They showed that periodic components are either finite, or infinite tubes, i.e. components of the form $\mathbb{Z}[\mathbb{A}_\infty]/\tau^m$. For $r=0$ we can boil these possibilities down to only one:

\begin{thm}
Let $\Theta$ be a periodic component of the stable Auslander-Reiten quiver of $U_\zeta(G_r)$. Then $\Theta$ is an infinite tube. Assume further that $r=0$ and for $\charac k=0$ let $\ell$ be a good prime for $\Phi$ and $\ell>3$ for types $\mathbb{B}$ and $\mathbb{C}$ and $\ell\nmid n+1$ for type $\mathbb{A}_n$ and $\ell\neq 9$ for $\mathbb{E}_6$, for $\charac k=p>0$ let $p$ be good for $\Phi$ and $\ell\geq h$. Then $\Theta$ is a homogeneous tube, i.e. $\Theta\cong \mathbb{Z}[\mathbb{A}_\infty]/\tau$.
\end{thm}

\begin{proof}
By \cite[p. 292, Theorem]{HPR80} $\Theta$ is either a finite component or an infinite tube. It cannot be a finite component since this is equivalent to the block containing $\Theta$ being representation-finite by a classical result of Auslander (cf. \cite[Theorem 5.4]{ASS06}). Furthermore for a block of $U_\zeta(G_r)$ being of finite representation type is the same as being semi-simple by our previous results in \cite{K11}. For $r=0$ even more holds since the even cohomology ring is generated in degree two by \cite[Theorem 1.2.3]{BNPP11} and \cite[Theorem 5.1.4]{Dru11}. By Proposition \ref{generalsupportfacts} this implies that the $\Omega$-period of every periodic module divides two. Since $U_\zeta(G_0)$ is symmetric, we have $\tau\cong \Omega^2$ and hence the $\tau$-period is one, i.e. the component is a homogenous tube.
\end{proof}

Our goal in the next sections will be to reduce the cases for non-periodic components that can occur.

\section{Kronecker components}

In this section we will classify (up to Morita equivalence) all self-injective algebras having a component of Kronecker type, i.e. a component isomorphic to $\mathbb{Z}[\tilde{\mathbb{A}}_{12}]$. To establish this result we need the following result due to Erdmann in the symmetric case, which generalizes to the self-injective case without difficulties. Note that a similar theorem due to Brenner and Butler \cite[Theorem 1.1]{BB98} states that this also holds even if the algebra is not self-injective. However in this case one has to assume that the algebra is tame.

\begin{prop}[{\cite[Theorem IV.3.8.3]{Er90}}, {\cite[Theorem 4.6, Proof]{Fa99b}}]
Let $A$ be a self-injective algebra. If $\Theta$ is a component of the stable Auslander-Reiten quiver isomorphic to $\mathbb{Z}[\tilde{\mathbb{A}}_{12}]$, then it is attached to a projective indecomposable module of length four. The lengths of the modules occurring in the non-stable component belonging to $\Theta$ are as follows:
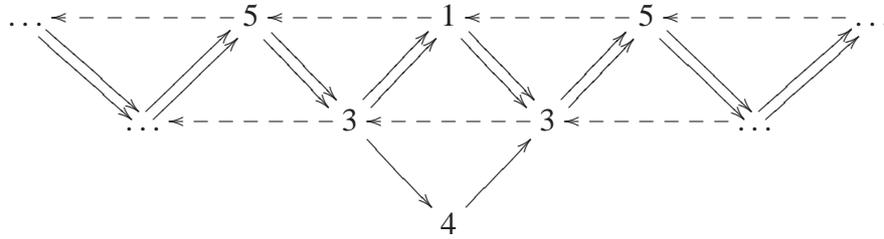
\[\begin{xy}
\xymatrix{
\dots\ar@<2pt>[rd]\ar@<-2pt>[rd]&&5\ar@<2pt>[rd]\ar@<-2pt>[rd]\ar@{-->}[ll]&&1\ar@<2pt>[rd]\ar@<-2pt>[rd]\ar@{-->}[ll]&&5\ar@<2pt>[rd]\ar@<-2pt>[rd]\ar@{-->}[ll]&&\dots\ar@{-->}[ll]\\
&\dots\ar@<2pt>[ru]\ar@<-2pt>[ru]&&3\ar@<2pt>[ru]\ar@<-2pt>[ru]\ar@{-->}[ll]\ar[rd]&&3\ar@<2pt>[ru]\ar@<-2pt>[ru]\ar@{-->}[ll]&&\dots\ar@<2pt>[ru]\ar@<-2pt>[ru]\ar@{-->}[ll]\\
&&&&4\ar[ru]
}
\end{xy}\]
The projective indecomposable module satisfies $\heart P\cong S\oplus S$ for some simple module $S$, where $\heart P=\rad{P}/\soc P$ is the heart of the projective module.
\end{prop}

By a more detailed analysis we now get to the classification of all algebras with Kronecker type components. All of them are special biserial, in particular tame (cf. \cite{WW85}). A special biserial algebra is the following:

\begin{defn}
An algebra $A$ is called \emphbf{special biserial} if $A$ is Morita equivalent to a path algebra with relations $kQ/I$, where $Q$ and $I$ satisfy the following:
\begin{enumerate}
\item[(SB 1)] In any vertex of $Q$ there start at most two arrows.
\item[(SB 1')] In any vertex of $Q$ there end at most two arrows.
\item[(SB 2)] If $a,b,c$ are arrows of $Q$ with $t(a)=s(b)=s(c)$, then one has $ba\in I$ or $ca\in I$.
\item[(SB 2')] If $a,b,c$ are arrows of $Q$ with $s(a)=t(b)=t(c)$, then one has $ab\in I$ or $ac\in I$.
\end{enumerate}
\end{defn}

\begin{thm}
\begin{enumerate}
\item If $A$ is a self-injective connected algebra with a component of the stable Auslander-Reiten quiver isomorphic to $\mathbb{Z}[\tilde{\mathbb{A}}_{12}]$, then the algebra is Morita equivalent to one of the following basic algebras with quiver
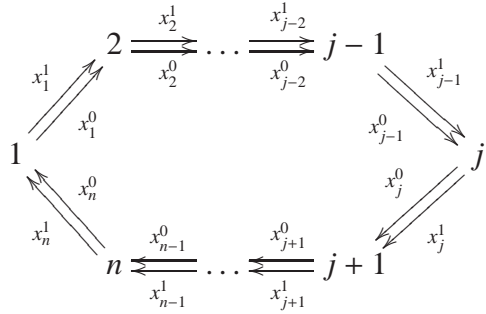
\[\begin{xy}
\xymatrix{
&2\ar@<2pt>[r]^{x^1_2}\ar@<-2pt>[r]_{x_2^0}&\dots\ar@<2pt>[r]^{x^1_{j-2}}\ar@<-2pt>[r]_{x^0_{j-2}}&{j-1}\ar@<2pt>[rd]^{x^1_{j-1}}\ar@<-2pt>[rd]_{x^0_{j-1}}\\
1\ar@<2pt>[ru]^{x^1_1}\ar@<-2pt>[ru]_{x^0_1}&&&& j\ar@<2pt>[ld]^{x^1_j}\ar@<-2pt>[ld]_{x^0_j}\\
&n\ar@<2pt>[lu]^{x^1_{n}}\ar@<-2pt>[lu]_{x^0_{n}}&\dots\ar@<2pt>[l]^{x^1_{n-1}}\ar@<-2pt>[l]_{x^0_{n-1}}&{j+1}\ar@<2pt>[l]^{x^1_{j+1}}\ar@<-2pt>[l]_{x^0_{j+1}}
}
\end{xy}\]
    and relations for each $j$ given by either $x_{j+1}^ix_j^i-\lambda_j^ix_{j+1}^{i+1}x_{j}^{i+1}$ and $x_{j+1}^ix_j^{i+1}$ for any $i$ or by $x_{j+1}^{i+1}x_j^i-\lambda_j^ix_{j+1}^ix_j^{i+1}$ and $x_{j+1}^ix_j^i$ for any $i$, where $\lambda_j^i\in k^\times$ for all $i,j$. (Many of these algebras are of course isomorphic.) Here the indices are added modulo $2$ and modulo $n$ respectively.
    In particular, all these algebras are special biserial. The Auslander-Reiten quiver consists of $n$ components isomorphic to $\mathbb{Z}[\tilde{\mathbb{A}}_{12}]$ each containing exactly one simple module and $n$ $\mathbb{P}^1(k)$-families of homogeneous tubes.
\item If $A$ is weakly-symmetric, then there are at most two simple modules. If there are two of them the following basic algebras remain:
\[\begin{xy}
\xymatrix{
1\ar@<2.5pt>[r]|{y_1}\ar@<7.5pt>[r]|{x_1}&2\ar@<2.5pt>[l]|{x_2}\ar@<7.5pt>[l]|{y_2}
}
\end{xy}\]
with relations: $x_2y_1=y_2x_1=0$ and $x_2x_1-y_2y_1=0$ and one of the following possibilities:
\begin{enumerate}[a)]
\item $x_1y_2=y_1x_2=0$ and $x_1x_2-\lambda y_1y_2=0$, where $\lambda\in k^\times$, say $A_2(\lambda)$;
\item $x_1x_2=y_1y_2=0$ and $x_1y_2-y_1x_2=0$, say $A_2(0)$.
\end{enumerate}
We have that $A_2(\mu)\cong A_2(\mu')$ iff $\mu'=\mu^{\pm 1}$.\\
If there is only one simple module then the following basic algebras remain:
\[\begin{xy}
\xymatrix{
1\ar@(ur,dr)[]^y\ar@(ul,dl)[]_x
}
\end{xy}\]
with relations $xy=yx=0$ and $x^2=y^2$, say $A_1(0)$, or $x^2=y^2=0$ and $xy=\lambda yx$, where $\lambda\in k^\times$, say $A_1(\lambda)$. For $\charac k=2$ we have $A_1(0)\ncong A_1(1)$. For $\charac k\neq 2$ we have $A_1(\mu)\cong A_1(\mu')$ iff $\mu'=\mu^{\pm 1}$ or $\{\mu,\mu'\}=\{0,1\}$.
\item The only symmetric algebras among these are $A_2(1)$ and $A_1(0)$ and $A_1(1)$.
\end{enumerate}
\end{thm}

\begin{proof}
Let $P_1$ be the projective indecomposable module attached to $\Theta$. Set $S_1:=\Kopf P_1$, $S_3:=\soc P_1$ and let $S_2$ be the simple module, such that $\heart P_1\cong S_2\oplus S_2$. Now define $P_1,\dots P_n$ by induction since $\Omega^i \Theta$ is also a component of Kronecker type attached to the projective cover of $S_{i+1}$, that is $P_{i+1}$. By general theory the Auslander-Reiten sequence attached to $P_i$ is $0\to \rad P_i\to \heart P_i\oplus P_i\to P_i/\soc P_i\to 0$ and $P_i$ was defined in such a way that $\heart P_i=S_{i+1}\oplus S_{i+1}$. Since there are only finitely many simple module, there is $l,m$ such that $S_l\cong S_{m+1}$. Choose them such that $|l-m|$ is minimal. Without loss of generality $l=1$. Then the modules $S_1,\dots,S_m$ form a complete list of composition factors of all the projective modules attached to the $\Omega^j\Theta$ for all $j$. Therefore they form a block. As the algebra is connected, these are all the simple modules . The quiver of $A$ can now be constructed by choosing a basis of $\rad(P_i)/\rad^2(P_i)$ for all $i$. Lifting this base to $P_i$ the relations can easily be determined. This implies that the algebra is special biserial. Thus by \cite{WW85} the Auslander-Reiten quiver of $A$ consists of $n$ components isomorphic to $\mathbb{Z}[\tilde{\mathbb{A}}_{12}]$ and $n$ $\mathbb{P}^1(k)$-families of homogeneous tubes.\\
If the algebra is weakly symmetric one has $S_1\cong S_3$. Direct construction of isomorphisms yields the stated quivers and relations. These isomorphisms can be given by sending the arrows to scalar multiples of themselves (Here we use that the field is algebraically closed.) To show that they are indeed non-isomorphic in the stated cases we use the fact that the trace of the Nakayama automorphism provides an invariant for the isomorphism class of an algebra. To construct a Frobenius homomorphism one can e.g. use \cite[Proposition 3.1]{HZ08}. For $A_2(\lambda)$ the trace of the Nakayama automorphism is $6+\lambda+\lambda^{-1}$, for $A_2(0)$ this invariant is $4$, the result for $A_1(\lambda)$ is $2+\lambda+\lambda^{-1}$.\\
Explicit calculation shows which of the algebras are symmetric.
\end{proof}

\begin{thm}
Let $\mathfrak{g}$ be a finite dimensional complex simple Lie algebra. Let $\ell>1$ be an odd integer not divisible by $3$ if $\Phi$ is of type $\mathbb{G}_2$. Furthermore for $\charac k=0$ assume that $\ell$ is good for $\Phi$ and that $\ell>3$ if $\Phi$ is of type $\mathbb{B}_n$ or $\mathbb{C}_n$. For $\charac k=p>0$ assume that $p$ is good for $\Phi$ and that $\ell>h$. If $U_\zeta(G_r)$ satisfies (fg), then the algebra $U_\zeta(G_r)$ has Kronecker components only if $\mathfrak{g}\cong \mathfrak{sl}_2$ and $r=0$ or $r=1$. For $r=1$ they belong to the image of the block embedding $F:\modu \Dist(\SL(2)_1)\to \modu U_\zeta(SL(2)_1), V\mapsto V\otimes \St_\ell$ (see \cite{K11}).
\end{thm}

\begin{proof}
As a component of Kronecker type belongs to a tame block, $U_\zeta(G_r)$ has to have a tame block. This is only possible in the stated cases by \cite[Proposition 5.6]{K11}. The known representation theory of $U_\zeta(\SL(2)_0)$ and $\Dist(\SL(2)_1)$ yields the result.
\end{proof}

That there is a classification in the Kronecker case is quite special. However one can prove partial results also for the case $\tilde{\mathbb{A}}_2$:

\begin{prop}
Let $A$ be a self-injective algebra with a component $\Theta$ of the stable Auslander-Reiten quiver isomorphic to $\mathbb{Z}[\tilde{\mathbb{A}}_2]$. Then there is a unique projective module $P_\Theta$ attached to the component that satisfies $\heart P_\Theta=S_\Theta\oplus M_\Theta$ with  an irreducible map $M_\Theta\twoheadrightarrow S_\Theta$ or $S_\Theta\hookrightarrow M_\Theta$. Furthermore $\Omega\Theta\neq \Theta$ and if $M_\Theta\twoheadrightarrow S_\Theta$ is irreducible, then there is an irreducible map $S_{\Omega\Theta}\hookrightarrow M_{\Omega\Theta}$, and if $S_\Theta\hookrightarrow M_\Theta$is irreducible, then there is an irreducible map $M_{\Omega\Theta}\twoheadrightarrow S_{\Omega\Theta}$.
\end{prop}

\begin{proof}
A component of type $\mathbb{Z}[\tilde{\mathbb{A}}_2]$ looks as follows (since it is necessarily attached to a projective module by \cite[p. 155]{BR87}, we have drawn that too):
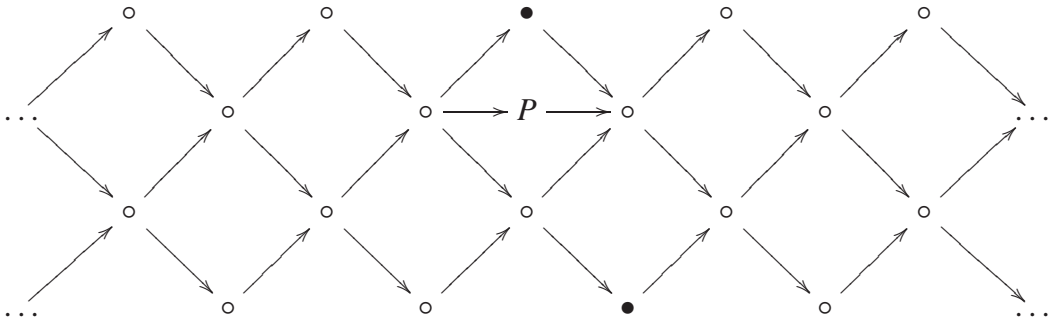
\[
\begin{xy}
\xymatrix{
&\circ\ar[rd]&&\circ\ar[rd]&&\bullet\ar[rd]&&\circ\ar[rd]&&\circ\ar[rd]\\
\dots\ar[ru]\ar[rd]&&\circ\ar[ru]\ar[rd]&&\circ\ar[ru]\ar[rd]\ar[r]&P\ar[r]&\circ\ar[ru]\ar[rd]&&\circ\ar[ru]\ar[rd]&&\dots\\
&\circ\ar[ru]\ar[rd]&&\circ\ar[ru]\ar[rd]&&\circ\ar[ru]\ar[rd]&&\circ\ar[ru]\ar[rd]&&\circ\ar[ru]\ar[rd]\\
\dots\ar[ru]&&\circ\ar[ru]&&\circ\ar[ru]&&\bullet\ar[ru]&&\circ\ar[ru]&&\dots
}
\end{xy}\]
Here the two bullets are identified (resp. their $\tau$-shifts). The standard almost split sequence attached to $P$ is $0\to \rad{P}\to \heart P\oplus P\to P/\soc P\to 0$. Therefore the irreducible maps originating in the predecessor of $P$ are surjective (except for the map $\rad P\to P$) while the irreducible maps terminating in the successor of $P$ (except for the map $P\to P/\soc P$) are injective. Since an irreducible map is either injective or surjective, an induction implies that the other meshes do not contain projective vertices and furthermore the maps "parallel" to these maps are injective (resp. surjective).\\
As $\Omega^{-1}\Theta$ satisfies the same properties, the component necessarily contains a simple module. Thus we have that this has to be at a vertex belonging to $\heart P$ since in all other vertices there either end injective irreducible maps or there start surjective irreducible maps. Depending on whether the simple module belongs to the bullet or to the other vertex we either have $M\twoheadrightarrow S$ or $S\hookrightarrow M$. Without loss of generality suppose the first possibility, otherwise dual arguments will yield the result.\\
Suppose that $\Omega\Theta=\Theta$. Then $\Omega S$ and $\Omega M$ are predecessors of $S\cong \Omega (P/\soc P)$ (since this is the only simple module in this component). Hence there are two possibilities. Either $\Omega M\cong M$, a contradiction since for a Euclidean component $\Delta$, we have that $\cx \mathbb{Z}[\Delta]=2$ by \cite[Proposition 1.1]{KZ11} or $\Omega M\cong \rad P$, a contradiction, since the two modules have different dimension.\\
Therefore $\soc P=: T\ncong S$. Since $\Omega \Theta$ is again a component isomorphic to $\mathbb{Z}[\tilde{\mathbb{A}}_2]$ we have that one of the two predecessors of $T$, $\Omega S$ or $\Omega M$ is isomorphic to $\rad{Q}$, where $Q$ is the projective module attached to $\Omega\Theta$. If $\Omega S\cong \rad{Q}$, then $Q$ is the projective cover of $S$. Thus we have the following part of $\Omega\Theta$:
\[\begin{xy}
\xymatrix{
\Omega S\ar[r]\ar[rd]&Q\ar[r]&Q/\rad{Q}\\
&T\ar[rd]\ar[ru]\\
\Omega M\ar[ru]\ar[rd]&&N\\
&\Omega S\ar[ru]
}
\end{xy}\]
where $N$ is the other direct summand of $\heart Q$. Thus there is an inclusion $T\hookrightarrow N$. Otherwise $\Omega S\cong N$, i.e. $N\cong \rad P(S)$, a contradiction to the fact that $P(S)$ cannot be attached to the component.
\end{proof}

\section{Restriction functors}

For $\alpha\in \Phi^+$ let $u_\zeta(f_\alpha)$ be the subalgebra of $U_\zeta(G_0)$ generated by $F_\alpha$. We start this section by proving that $U_\zeta(G_0):u_\zeta(f_\alpha)$ is a Frobenius extension. Our approach is similar to the approach by Farnsteiner and Strade for modular Lie algebras in \cite{FS91} and Bell and Farnsteiner for Lie superalgebras in \cite{BF93}. We use that the small quantum group is also a quotient of the De Concini-Kac form of the quantum group $\mathcal{U}_k(\mathfrak{g})$ (see e.g. \cite{DruPhD} for a definition).

\begin{defn}
Let $R$ be a ring, $S\subseteq R$ be a subring of $R$ and $\gamma$ be an automorphism of $S$. If $M$ is an $S$-module denote by $M^{(\gamma)}$ the $S$-module with the new action defined by $s*m=\gamma(s)m$. We say that $R$ is a \emphbf{free $\gamma$-Frobenius extension}\index{Frobenius extension} of $S$ if
\begin{enumerate}[(i)]
\item $R$ is a finitely generated free $S$-module, and
\item there exists an isomorphism $\varphi: R\to \Hom_S(R,S^{(\gamma)})$ of $(R,S)$-bimodules
\end{enumerate}
\end{defn}

For general theory on Frobenius extensions we refer the reader to \cite{NT60} and \cite{BF93}. We use the following definition:

\begin{defn}
Let $R$ be a ring, $S$ be a subring of $R$ and $\gamma $ be an automorphism of $S$.
\begin{enumerate}[(i)]
\item A \emphbf{$\gamma$-associative form} from $R$ to $S$ is a biadditive map $\langle -,-\rangle_\pi: R\times R\to S$, such that:
\begin{itemize}
\item $\langle sx,y\rangle_\pi=s\langle x,y\rangle_\pi$
\item $\langle x,ys\rangle_\pi=\langle x,y\rangle_\pi \gamma(s)$
\item $\langle xr,y\rangle_\pi=\langle x,ry\rangle_\pi$
\end{itemize}
for all $s\in S, r,x,y\in R$.
\item Let $\langle-,-\rangle_\pi:R\times R\to S$ be a $\gamma^{-1}$-associative form. Two subsets $\{x_1,\dots,x_n\}$ and $\{y_1,\dots,y_n\}$ of $R$ are said to form a \emphbf{dual free pair} relative to $\langle-,-\rangle_\pi$ if
\begin{itemize}
\item $R=\sum_{i=1}^{n}Sx_i=\sum_{i=1}^{n}y_iS$
\item $\langle x_i,y_j\rangle_\pi=\delta_{ij}$ for $1\leq i,j\leq n$.
\end{itemize}
\end{enumerate}
\end{defn}

\begin{lem}[{\cite[Corollary 1.2]{BF93}}]
Let $S$ be a subring of $R$ and let $\gamma$ be an automorphism of $S$. Then the following statements are equivalent:
\begin{enumerate}[(1)]
\item $R$ is a free $\gamma$-Frobenius extension of $S$.
\item There is a $\gamma^{-1}$-associative form $\langle-,-\rangle_\pi$ from $R$ to $S$ relative to which a dual free pair $\{x_1,\dots,x_n\}$, $\{y_1,\dots,y_n\}$ exists.
\end{enumerate}
More precisely if $\langle-,-\rangle_\pi$ is a $\gamma^{-1}$-associative form relative to which a dual free pair exists, then an isomorphism $R\to \Hom_S(R,S^{(\gamma)})$ is given by $y\mapsto (x\mapsto \gamma(\langle x,y\rangle_\pi))$ with inverse $f\mapsto  \sum_{i=1}^{n}y_if(x_i)$.
\end{lem}

\begin{defn}
Let $\alpha\in \Phi^+$. Define $\mathcal{O}_k(\mathfrak{g},\alpha)$ to be the subalgebra of $\mathcal{U}_k(\mathfrak{g})$ generated by $F_\alpha$ and $E^\ell_{\beta}, F^\ell_{\beta'}, K^\ell_{\alpha'}$ with $\beta, \beta'\in \Phi^+$, $\alpha'\in \Pi$ and $\beta'\neq \alpha$.
\end{defn}

\begin{prop}\label{Frobeniusextension}
Let $\alpha\in \Pi$. Then $\mathcal{U}_k(\mathfrak{g})$ is a free $\gamma$-Frobenius extension of $\mathcal{O}_k(\mathfrak{g},\alpha)$ where $\gamma$ is given by $F_\alpha\mapsto \zeta^{-\sum_{\alpha'\in \Pi}(\alpha',\alpha)}F_\alpha$.
\end{prop}

\begin{proof}
We use the equivalent description of Frobenius extension in \cite[Corollary 1.2]{BF93}. The freeness follows from the Poincaré-Birkhoff-Witt theorem for quantum groups. We now construct a $\gamma$-associative form: Let $\mathcal{L}^n\in \mathbb{N}^n$ be the vector with $\mathcal{L}^n_i:=\ell-1$, $\mathcal{L}^N\in \mathbb{N}^N$ be the vector with $\mathcal{L}^N_i:=\ell-1$ and $\mathcal{L}^{N-1}\in \mathbb{N}^N$ be the vector with $\mathcal{L}^{N-1}_i:=\delta_{i,\alpha}(\ell-1)$. Let $\mathcal{O}'$ be the complement of $\mathcal{O}_k(\mathfrak{g},\alpha)F^{\mathcal{L}^{N-1}}K^{\mathcal{L}^n}E^{\mathcal{L}^{N}}$ with respect to the PBW basis, i.e. $\mathcal{U}_k(\mathfrak{g})=\mathcal{O}'\oplus \mathcal{O}_k(\mathfrak{g},\alpha)F^{\mathcal{L}^{N-1}}K^{\mathcal{L}^n}E^{\mathcal{L}^{N}}$. Let $\pi:\mathcal{U}_k(\mathfrak{g})\to \mathcal{O}_k(\mathfrak{g},\alpha)$ be the projection onto the second summand.
Define $\langle x,y\rangle_\pi:=\pi(xy)$. Then we obviously have $\pi(xy)=x\pi(y)$ for all $x\in \mathcal{O}_k(\mathfrak{g},\alpha), y\in \mathcal{U}_k(\mathfrak{g})$. Furthermore we claim that $\pi(yF_\alpha)=\zeta^{\sum_{\alpha'\in\Pi}(\alpha',\alpha)}\pi(y)F_\alpha$ for all $y\in \mathcal{U}_k(\mathfrak{g})$: It suffices to show this on the basis vectors $F^{a}K^{b}E^{c}$ for $a\in \{x\in \mathbb{N}^N|x_\alpha=0\}$, $b\in \mathbb{N}^n$ and $c\in \mathbb{N}^N$ and for the associated graded algebra $\gr \mathcal{U}_k(\mathfrak{g})$, where the reduced expression used to construct the convex ordering is chosen in such a way that $\alpha$ is a minimal root. This algebra is given by generators and relations in {\cite[Proposition 1.7]{DK90}}. It follows that $F^{a}K^{b}E^{c}F_\alpha=\zeta^{-(\sum b_i,\alpha)+(\sum a_i,\alpha)}F_\alpha F^{a}K^{b}E^{c}$ in $\gr\mathcal{U}_k(\mathfrak{g})$. So low-order terms will not have an influence on $\image(\pi)$. And
\[F^{\mathcal{L}^{N-1}}K^{\mathcal{L}^n}E^{\mathcal{L}^{N}}F_\alpha=\zeta^{(\ell-1)(\sum_{\alpha\neq \alpha'\in \Phi^+}(\alpha',\alpha)-\sum_{\alpha'\in \Pi}(\alpha',\alpha))}F_\alpha F^{\mathcal{L}^{N-1}}K^{\mathcal{L}^n}E^{\mathcal{L}^{N}}.\]
The coefficient is equal to
\[\zeta^{-(2\rho,\alpha)+(\alpha,\alpha)+\sum_{\alpha'\in \Pi}(\alpha',\alpha)}=\zeta^{\sum_{\alpha'\in \Pi}(\alpha',\alpha)}.\]
The dual free pair will be the PBW basis $\{F^{a}K^{b}E^{c}\}$ and the set $\{\zeta^{-c(a,b,c)}F^{\ell-1-a}K^{\ell-1-b}E^{\ell-1-c}\}$, where $c(a,b,c)$ is chosen in such a way that \[\pi(F^{\ell-1-a'}K^{\ell-1-b'}E^{\ell-1-c'}F^{a}K^{b}E^{c})=\zeta^{c(a,b,c)}\delta_{a,a'}\delta_{b,b'}\delta_{c,c'},\] where $c(a,b,c)$ is some integer depending only on $a$, $b$ and $c$. This is possible by a similar computation as before.
\end{proof}

\begin{prop}\label{frobeniusextension}
Let $\alpha\in \Pi$. Then $U_\zeta(G_0)$ is a free $\gamma$-Frobenius extension of $u_\zeta(f_\alpha)$ where $\gamma$ is given by $F_\alpha\mapsto \zeta^{-\sum_{\alpha'\in \Pi}(\alpha',\alpha)}F_\alpha$ for each root subalgebra.
\end{prop}

\begin{proof}
The ideal $I=\langle E_\beta^\ell, F_\beta^\ell, K_\alpha^\ell-1|\beta\in \Phi^+, \alpha\in \Pi\rangle$ satisfies the conditions of \cite[Theorem 2.3]{Fa96} (i.e.  $\mathcal{U}_k(\mathfrak{g}):\mathcal{O}_k(\mathfrak{g},\alpha)$ being a $\gamma$-Frobenius extension and $I\subseteq \mathcal{U}_k(\mathfrak{g})$ an ideal with $\pi(I)\subseteq I\cap S$ and $\gamma(I\cap S)\subseteq I\cap S)$) since $\pi(I)=0$ and $\gamma$ is on generators just given by scalar multiplication. Hence $\mathcal{U}_k(\mathfrak{g})/I:\mathcal{O}_k(\mathfrak{g},\alpha)/\mathcal{O}_k(\mathfrak{g},\alpha)\cap I$ is a $\gamma'$-Frobenius extension, where $\gamma'$ is induced by $\gamma$.
\end{proof}

\begin{thm}
$U_\zeta(G_0)$ is a free $\gamma$-Frobenius extension of $u_\zeta(f_\beta)$, where $\beta\in \Phi^+$ and $\gamma$ as in the foregoing proposition.
\end{thm}

\begin{proof}
Let $T$ be one of Lusztig's $T$-automorphisms (see e.g. \cite[8.6]{Jan96}). Then there is an isomorphism
$\Hom_{u_\zeta(f_\alpha)}(U_\zeta(G_0),u_\zeta(f_\alpha)^{(\gamma)})\to \Hom_{u_\zeta(f_\beta)}(U_\zeta(G_0),u_\zeta(f_\beta)^{(\gamma)})$ given by $f\mapsto TfT^{-1}$. Induction on the length of $\beta$ now gives the result.
\end{proof}

The second situation we want to consider in this section is compatibly graded modules, i.e. modules for $U_\zeta(G_0)U_\zeta^0(\mathfrak{g})$. We show that they are the $\mathbb{Z}^n$-compatibly gradable modules among the modules for $U_\zeta(G_0)$. This allows us to use the following result, which was first obtained in the case $n=1$ by Gordon and Green. The stated generalization to arbitrary $n$ uses a result by Gabriel.

\begin{thm}[{\cite[Corollary 1.3, Proposition 1.4 (2), Theorem 2.3]{FarLN}}, cf. {\cite[Corollaire IV.4.4]{G62}}, cf. {\cite[Theorem 3.5]{GG82}}]\label{gradedAR}
Let $A$ be a $\mathbb{Z}^n$-graded algebra. Then the category of graded modules admits almost split sequences and the forgetful functor from finite dimensional graded modules to finite dimensional modules sends indecomposables to indecomposables, projectives to projectives and almost split sequences to almost split sequences.
\end{thm}

Using this we now proceed to prove statements about modules for the algebra $U_\zeta(G_0)U_\zeta^0(\mathfrak{g})$ in a similar way as Farnsteiner did for restricted enveloping algebras in \cite{Fa05}.

\begin{lem}\label{res}
\begin{enumerate}[(i)]
\item The category of finite dimensional modules over $U_\zeta(G_0)U_\zeta^0(\mathfrak{g})$ is a sum of blocks for the category $\modu U_\zeta(G_0)\#U_\zeta^0(\mathfrak{g})$.
\item The category of finite dimensional $U_\zeta(G_0)U_\zeta^0(\mathfrak{g})$-modules has almost split sequences.
\item The canonical restriction functor $\modu U_\zeta(G_0)U_\zeta^0(\mathfrak{g})\to \modu U_\zeta(G_0)$ sends indecomposables to indecomposables and almost split sequences to almost split sequences.
\end{enumerate}
\end{lem}

\begin{proof}
The map $U_\zeta(G_0)\# U_\zeta^0(\mathfrak{g})\to U_\zeta(G_0)U_\zeta^0(\mathfrak{g}), u\# v\mapsto uv$ is surjective and its Hopf kernel is  $U_\zeta(G_0)\cap U_\zeta^0(\mathfrak{g})$ given by $K_\alpha\mapsto K_\alpha\# K_\alpha^{-1}$. The following computation shows that via this embedding $U_\zeta^0(G_0)=U_\zeta(G_0)\cap U_\zeta^0(\mathfrak{g})$ lies in the center of $U_\zeta(G_0)\# U_\zeta^0(\mathfrak{g})$:
\begin{align*}
(K_\alpha\# K_\alpha^{-1})(F_\beta\# u)&=K_\alpha K_\alpha^{-1}(F_\beta)\#K_\alpha^{-1}u=K_\alpha K_\alpha^{-1}F_\beta K_\alpha\#uK_\alpha^{-1}\\
&=F_\beta uK_\alpha u^{-1}\#uK_\alpha^{-1}=(F_\beta\# u)(K_\alpha\# K_\alpha^{-1}),
\end{align*}
where $u\in U_\zeta^0(\mathfrak{g})$, similarly for $E_\beta$ and $K_\alpha$. Thus $U_\zeta^0(G_0)$-weight spaces are $U_\zeta(G_0)\# U_\zeta^0(\mathfrak{g})$-submodules. So $U_\zeta^0(G_0)$ operates on indecomposable modules by a single character and using Krull-Remak-Schmidt we get
\[\modu U_\zeta(G_0)\# U_\zeta^0(\mathfrak{g}) =\bigoplus_{\lambda\in \ch(U^0_\zeta(G_0))} (\modu(U_\zeta(G_0)\# U_\zeta^0(\mathfrak{g})))_\lambda\]
with $\modu U_\zeta(G_0)U_\zeta^0(\mathfrak{g})= (\modu U_\zeta(G_0)\#U_\zeta^0(\mathfrak{g}))_0$ as $U_\zeta^0(G_0)$ is semisimple. If $V$, $W$ are finite dimensional simple modules giving rise to different characters $\lambda\neq \mu$, then
\[\Ext^1_{U_\zeta(G_0)\# U_\zeta^0(\mathfrak{g})}(V,W)\cong H^1(U_\zeta(G_0)U_\zeta^0(\mathfrak{g}), \Hom_{U_\zeta^0(G_0)}(V,W))\cong 0,\]
by \cite[I.4.1 (3)]{Jan03}. Thus they belong to different blocks. This shows (i).\\
(ii) and (iii) follow from the fact that $\modu(U_\zeta(G_0)\# U_\zeta^0(\mathfrak{g}))$ coincides with the category of $\mathbb{Z}^n$-graded $U_\zeta(G_0)$-modules with homomorphisms in degree $0$ as the character group of $U_\zeta^0(\mathfrak{g})$ is $\mathbb{Z}^n$ . Under this identification the restriction functors $\modu(U_\zeta(G_0)\# U_\zeta^0(\mathfrak{g}))\to \modu U_\zeta(G_0)$ and $\modu_{\mathbb{Z}^n} U_\zeta(G_0)\to \modu U_\zeta(G_0)$  coincide. So these statements follow from Theorem \ref{gradedAR}.
\end{proof}

The following statement was first obtained by Drupieski in \cite[Lemma 3.3]{Dru10} by imitating the proof for the corresponding statement for algebraic groups. For us it is just a corollary of the foregoing two statements. Recall that a module $M$ is called rational iff it admits a weight space decomposition, i.e. $M=\bigoplus_{\lambda\in \ch(U_\zeta(\mathfrak{g}))} M_\lambda$.

\begin{cor}
Let $M$ be a finite dimensional $U_\zeta(G_0)U_\zeta^0(\mathfrak{g})$-module. Then the following statements are equivalent:
\begin{enumerate}[(1)]
\item $M$ is a rationally injective $U_\zeta(G_0)U_\zeta^0(\mathfrak{g})$-module.
\item $M$ is an injective $U_\zeta(G_0)$-module.
\item $M$ is a projective $U_\zeta(G_0)$-module.
\item $M$ is a rationally projective $U_\zeta(G_0)U_\zeta^0(\mathfrak{g})$-module.
\end{enumerate}
\end{cor}

\begin{lem}
The restriction functor from the last lemma induces a homomorphism \[\mathcal{F}:\Gamma_s(U_\zeta(G_0)U_\zeta^0(\mathfrak{g}))\to \Gamma_s(U_\zeta(G_0))\]
of stable translation quivers and components are mapped to components via this functor.
\end{lem}

\begin{proof}
Thanks to the foregoing corollary  a module is projective for $U_\zeta(G_0)$ if and only if it is rationally projective for $U_\zeta(G_0)U_\zeta^0(\mathfrak{g})$. The forgetful functor obviously commutes with direct sums so the homomorphism follows from Lemma \ref{res} and the fact that the stable Auslander-Reiten quiver can be defined via Auslander-Reiten sequences in the following way: For an Auslander-Reiten sequence $0\to M\to \bigoplus E_i^{n_i}\to N\to 0$ there are $n_i$ arrows $M\to E_i$ in $\Gamma_s(A)$.\\
Let $\Theta$ be a component of $\Gamma_s(U_\zeta(G_0)U^0_\zeta(\mathfrak{g}))$. Since $\mathcal{F}$ is a homomorphism of stable translation quivers, there exists a unique component $\Psi\subseteq \Gamma_s(U_\zeta(G_0))$ with $\mathcal{F}(\Theta)\subseteq \Psi$. As $\mathcal{F}(\Theta)$ is $\tau_{U_\zeta(G_0)}$-invariant, we only have to show that each neighbour of an element of $\mathcal{F}(\Theta)$ also belongs to $\mathcal{F}(\Theta)$. To that end, we consider an isomorphism class $[M]\in \Theta$ as well as the almost split sequence $0\to \tau_{U_\zeta(G_0)}(\mathcal{F}(M))\to E\to \mathcal{F}(M)\to 0$. We decompose $E=\bigoplus_{i=1}^n E_i^{m_i}$ into indecomposable modules, so that the distinct isomorpism classes $[E_i]$ are the predecessors of $[\mathcal{F}(M)]\in \Psi$. We next consider the almost split sequence $0\to \tau_{U_\zeta(G_0)U_\zeta^0(\mathfrak{g})}(M)\to X\to M\to 0$ terminating in $M$. Thanks to Lemma \ref{res} the almost split sequence terminating in $\mathcal{F}(M)$ is isomorphic to $0\to \tau_{U_\zeta(G_0)}(\mathcal{F}(M))\to \mathcal{F}(X)\to \mathcal{F}(M)\to 0$. In particular, if $X=\bigoplus_{j=1}^m X_j^{r_j}$ is the decomposition of $X$ into its indecomposable constituents, then $E\cong \mathcal{F}(X)\cong \bigoplus_{j=1}^m \mathcal{F}(X_j)^{r_j}$ is the corresponding decomposition of $\mathcal{F}(X)$. Thus the Theorem of Krull-Remak-Schmidt implies that for each $i=1,\dots, n$ there exists $i_j\in \{1,\dots,m\}$, such that $[E_i]=[\mathcal{F}(X_{i_j})]$. Consequently, each isoclass $[E_i]$ with $E_i$ non-projective belongs to $\mathcal{F}(\Theta)$ as desired. Using the bijectivity of $\tau_{U_\zeta(G_0)}$ one proves the corresponding statement for the successors of vertices belonging to $\mathcal{F}(\Theta)$.
\end{proof}

\section{Components of complexity two}

In this section we rule out the case of components of the form $\mathbb{Z}[\Delta]$, where $\Delta$ is a Euclidean diagram, in Webb's Theorem and under additional assumptions also that of the infinite Dynkin tree classes except $\mathbb{A}_\infty$, i.e. we prove our main theorem. Our approach is similar to the approach taken for restricted enveloping algebras in \cite{Fa99} and \cite{Fa00}.

\begin{defn}
Let $A$ be a Frobenius algebra. Then $A$ is said to \emphbf{admit an analogue of Dade's Lemma} for a subcategory $\mathcal{M}$ of the module category if there is a family of subalgebras $\mathcal{B}$, such that $A:B$ is a $\beta$-Frobenius extension for some automorphism $\beta$ (possibly depending on $B$) and $B\cong k[X]/(X^m)$ for all $B\in\mathcal{B}$ and $M\in \mathcal{M}$ is projective iff $M|_B$ is projective for all $B\in \mathcal{B}$.
\end{defn}

Our basic example of an algebra admitting an analogue of Dade's Lemma is $U_\zeta(G_0)$ by a result of Drupieski (\cite[Theorem 4.1]{Dru10}).

\begin{thm}
Let $A$ be a Frobenius algebra admitting an analogue of Dade's Lemma for a subcategory $\mathcal{M}$ of the module category closed under kernels of irreducible morphisms. Let $\Theta\neq \mathbb{Z}[\mathbb{A}_\infty]$ be a non-periodic component of finite complexity of the stable Auslander-Reiten quiver of $A$ contained in  $\mathcal{M}$. Then the component is of complexity two.
\end{thm}

\begin{proof}
We have already seen that a non-periodic component $\Theta$ of finite complexity is of the form $\mathbb{Z}[\Delta]$, where $\Delta$ is a Euclidean or infinite Dynkin diagram. If $\Delta$ is Euclidean the statement was proven in \cite[Proposition 1.1]{KZ11}. So let $\Delta$ be of infinite Dynkin type, not $\mathbb{A}_\infty$. Choose $[E]\in \Theta$ to have two successors and two predecessors. For each $B\in\mathcal{B}$ denote the unique simple $B$-module by $S_B$ and consider the $A$-module $M_B:=A\otimes_B S_B$. We have $\Omega_B^2 S_B\cong S_B$ by the representation theory of $k[X]/X^n$. Applying the induction functor to a minimal projective resolution we get that there exist projective modules $P,Q$, such that $\Omega^2_A M_B\oplus P\cong M_B\cong \Omega^{-2}_A M_B\oplus Q$. As $S_B$ is a factor as well as a submodule of $B$, we have that $M_B$ is a factor as well as a submodule of $A$ as tensoring with $A$ is an exact functor since $A$ is a projective $B$-module (because $A:B$ is a Frobenius extension). So $\max \{ \dim M_B, \dim \Omega_A M_B, \dim \Omega^{-1}_A M_B\}\leq \dim A$.\\
Now choose $g:E\to N$ to be irreducible, where $N$ is indecomposable. By choice of $[E]$, $g$ is properly irreducible, i.e. it is not almost split. Any irreducible map is either injective or surjective. Suppose $g$ is surjective, then $M:=\ker g$ is not projective, so there is $B\in \mathcal{B}$, such that $M|_B$ is not projective. Since $A:B$ is a Frobenius extension, we have that $\underline{\Hom}_A(M,\Omega^{-1} M_B)\cong \Ext^1_A(M,M_B)\cong \Ext^1_B(M,S_B)$ by \cite[Lemma 7]{NT60}. But this is non-zero since $M$ is not $B$-projective. From \cite[Proposition 1.5]{Er95} it follows now that $\dim E\leq \dim N+\dim A$. In the case that $g$ is injective by a dual argument one obtains $\dim N\leq \dim E+\dim A$.\\
By choice of $[E]$, there exists a walk $\tau(E)\to N\to E$ in $\Theta$ given by properly irreducible maps. We therefore have $\dim \Omega^2_A(E)=\dim \tau(E)\leq \dim E+2\dim A$. Repeated application of $\tau$ gives $\dim \Omega^{2n}_A E\leq \dim E+ 2n\dim A$.\\
As the module $\Omega E$ satisfies the same properties as $E$, since $\Omega_A \Theta\cong \Theta$, we conclude that there is some $C>0$, such that $\dim \Omega^n_A E\leq Cn$ for all $n\geq 1$. Therefore $\cx E\leq 2$.
\end{proof}

In our situation this gives:

\begin{prop}\label{cx2}
Let $\mathfrak{g}$ be simple. Let $\Theta\neq \mathbb{Z}[\mathbb{A}_\infty]$ be a nonperiodic component of the stable Auslander-Reiten quiver of $U_\zeta(G_0)$ containing the restriction of a $U_\zeta(G_0)U_\zeta^0(\mathfrak{g})$-module. Then $\cx \Theta =2$.
\end{prop}

\begin{proof}
In the previous section we have shown that the algebras $u_\zeta(f_\alpha)$ satisfy $U_\zeta(G_0):u_\zeta(f_\alpha)$ being a Frobenius extension. By the same token the subcategory $\mathcal{M}$ of restrictions of $U_\zeta(G_0)U_\zeta^0(\mathfrak{g})$-modules is closed under kernels of irreducible maps. Therefore Drupieski's Theorem \cite[Theorem 4.1]{Dru10}) implies the result.
\end{proof}

For higher $r$ this approach does not seem to work although Drupieski's Theorem holds more general for higher $r$, but the algebras to which one restricts are no longer Nakayama.

\begin{ex}
Consider $U_\zeta({N_{\SL(2)}}_1)\cong k[X,Y]/(X^p, Y^\ell)$, an algebra to which Drupieski's Theorem restricts in the case $r=1$. Then this algebra does not admit an analogue of Dade's Lemma. Consider the module $M$ indicated by the following picture:
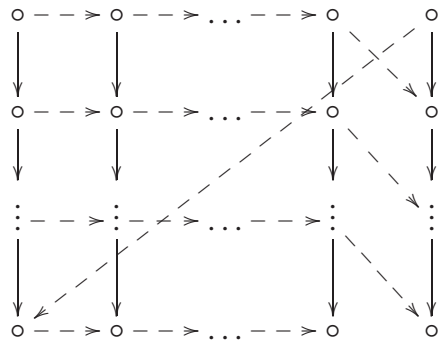
\[\begin{xy}
\xymatrix{
\circ\ar[d]\ar@{-->}[r]&\circ\ar[d]\ar@{-->}[r]&\dots\ar@{-->}[r]&\circ\ar[d]\ar@{-->}[rd]&\circ\ar[d]\ar@{-->}[dddllll]\\
\circ\ar[d]\ar@{-->}[r]&\circ\ar[d]\ar@{-->}[r]&\dots\ar@{-->}[r]&\circ\ar[d]\ar@{-->}[rd]&\circ\ar[d]\\
\vdots\ar[d]\ar@{-->}[r]&\vdots\ar[d]\ar@{-->}[r]&\dots\ar@{-->}[r]&\vdots\ar[d]\ar@{-->}[rd]&\vdots\ar[d]\\
\circ\ar@{-->}[r]&\circ\ar@{-->}[r]&\dots\ar@{-->}[r]&\circ&\circ
}
\end{xy}\]
The circles represent the basis vectors of the module. The different arrows stand for the action of $X$ and $Y$ respectively. It is easy to see that this module is projective for $k[X]/(X^p)$ and $k[Y]/(Y^\ell)$, but not projective for $k[X,Y]/(X^p,Y^\ell)$, since it has a two-dimensional socle and is $p\ell$-dimensional. For the other subalgebras $B$ it can also be checked that either $A:B$ is not a Frobenius extension or $M$ is projective for $B$.
\end{ex}

For the algebra $U_\zeta(G_0)$ the action of the algebraic group $G$ on the support variety of a module for the "big" quantum group $U_\zeta(\mathfrak{g})$ can be used to exclude components of complexity two.

\begin{thm}
Let $\charac k$ be odd or zero and good for $\Phi$. Assume $\ell$ is odd, coprime to three if $\Phi$ has type $\mathbb{G}_2$ and $\ell\geq h$. Let $M$ be a $U_\zeta(\mathfrak{g})$-module of complexity two. Then the following statements hold:
\begin{enumerate}[(i)]
\item There exist normal subgroups $K,H\subset G$ such that
\begin{enumerate}[(a)]
\item $G=HK$
\item $\mathfrak{g}=\Lie(H)\oplus \Lie(K)$, and
\item $\Lie(K)\cong \mathfrak{sl}_2$.
\end{enumerate}
\item If $M$ is indecomposable, then $\mathcal{V}_{U_\zeta(G_0)}(M)=\mathcal{V}_{U_\zeta(K_0)}(M)$ and $M$ is projective as a $U_\zeta(H_0)$-module.
\end{enumerate}
\end{thm}

\begin{proof}
By general theory (see e.g. \cite[Theorem 8.1.5]{Spr98}) we have $G=G^1\cdots G^r$, where the $G^i$ are the minimal non-trivial closed, connected, normal subgroups of $G$. Since $M$ is a $U_\zeta(\mathfrak{g})$-module, the support variety $\mathcal{V}_{U_\zeta(G_0)}(M)$ and its projectivization $\mathbb{P}\mathcal{V}_{U_\zeta(G_0)}(M)$ are stable under the adjoint representation. Owing to Borel's Fixed Point Theorem (see e.g. \cite[Theorem 6.2.6]{Spr98}), there exists $[x_0]\in \mathbb{P}\mathcal{V}_{U_\zeta(G_0)}(M)$ that is fixed by a given Borel subgroup $B\subset G$. Consequently the stabilizer $P_0$ of $[x_0]$ is parabolic.\\
According to the orbit formula we have $\dim G-\dim P_0=\dim G\cdot [x_0]\leq 1$. Therefore there are two possibilities for this dimension, either $\dim G=\dim P_0$, which implies $P_0=G$, or $\dim G-\dim P_0=1$.\\
The assumption $P_0=G$ implies that $kx_0\subseteq \mathfrak{g}$ is invariant under the adjoint representation. Hence, $kx_0$ is an abelian ideal of $\mathfrak{g}$. As $\mathfrak{g}$ is semisimple, it follows that $x_0=0$, a contradiction. Thus we have $\dim P_0=\dim G-1$. By the arguments of \cite[Proposition 13.13]{Bor91} (see also \cite[Proof of Proposition 5.1]{Fa00}) we see that the action of $G$ on $G/P_0$ induces a surjective homomorphism $\varphi: G\to \PGL_2$ of algebraic groups. As the $\varphi(G^i)$ are normal subgroups of the simple group $\PGL_2$, we have that $\varphi(G^i)=\PGL_2$ or $G_i\subseteq \ker \varphi$. Since $(G^i,G^j)=1$, we also have that $(\varphi G^i, \varphi G^j)=1$, so if there were two indices $i,j$ with $\varphi(G^i)=\varphi(G^j)=\PGL_2$, then $\varphi(G^j)\in Z(\PGL_2)$, a contradiction. Thus there exists exactly one index $i_0\in \{1,\dots,r\}$ such that $\varphi(G^{i_0})=\PGL_2$, without loss of generality $i_0=1$. Setting $K:=G^1$ and $H:=G^2\cdots G^r$ we have $\dim K=3$ and $H\subseteq \ker \varphi\subseteq P_0$. Since $(G^i, G^j)=1$ for $i\neq j$ we have $\Lie(H)\subseteq C_{\mathfrak{g}}(\Lie(K))$. Since $\Lie(K)$ has only inner derivations it follows that for $x\in \mathfrak{g}$ the derivation $\ad x|_{\Lie(K)}$ is inner, i.e. there exists $v\in \Lie(K)$, such that $x-v\in C_{\mathfrak{g}}(\Lie(K))$. Therefore since the center of $\Lie(K)$ is trivial, it follows that $\mathfrak{g}=\Lie(K)\oplus C_{\mathfrak{g}}(\Lie(K))$. Moreover, $H\cap K$ is finite, so that
\[\dim \mathfrak{g}=\dim G=\dim H+\dim K=\dim \Lie(H)+\dim \Lie(K).\]
This shows that $\mathfrak{g}=\Lie(K)\oplus \Lie(H)$.\\
Let $M$ now be indecomposable. Since $\dim \mathcal{V}_{U_\zeta(G_0)}(M)=2$, we have $\dim G\cdot [x]\leq 1$ for all $[x]\in \mathbb{P}\mathcal{V}_{U_\zeta(G_0)}(M)$. If $\dim G\cdot [x]=0$ for some $x$, then $G[x]=[x]$ as $\mathbb{P}\mathcal{V}_{U_\zeta(G_0)}$ is connected, and $kx\neq 0$ is an abelian ideal of $\mathfrak{g}$, a contradiction. Hence all orbits have dimension one, so that each of them is closed and thus an irreducible component of $\mathbb{P}\mathcal{V}_{U_\zeta(G_0)}(M)$. Since the orbits do not intersect and $\mathbb{P}\mathcal{V}_{U_\zeta(G_0)}(M)$ is connected, it follows that $\mathbb{P}\mathcal{V}_{U_\zeta(G_0)}(M)$ is irreducible. In particular we have
\[\mathbb{P}\mathcal{V}_{U_\zeta(G_0)}(M)=G\cdot [x_0]=K\cdot H\cdot [x_0]=K\cdot [x_0],\]
since $H\subseteq P_0$. Since $H$ and $K$ commute, this implies that $H$ operates trivially on $\mathbb{P}\mathcal{V}_{U_\zeta(G_0)}(M)$.\\
Let $x$ be a nonzero element of $\mathcal{V}_{U_\zeta(G_0)}(M)$. By the observation above there exists a character $\alpha_x:H\to k^\times$, such that $\Ad(h)(x)=\alpha_x(h)x$ for all $h\in H$. Since $G^i=(G^i,G^i)$ for all $2\leq i\leq r$ and commutators are mappped to commutators via $\alpha_x$ it readily follows that $\alpha_x(h)=1$ for $h\in H$. Consequently $[\Lie(H),x]=0$. Writing $x=y+z$, where $y\in \Lie(K)$ and $z\in \Lie(H)$, we have $[\mathfrak{g},z]=[\Lie(H),z]=[\Lie(H),x]=0$. Hence $z\in C(\mathfrak{g})$ defines an abelian ideal $kz\subset \mathfrak{g}$. Since $\mathfrak{g}$ is semisimple, we conclude that $z=0$. Consequently $x=y\in \Lie(K)$. As a consequence of the K\"unneth formula we now obtain $\mathcal{V}_{U_\zeta(G_0)}(M)=\mathcal{V}_{U_\zeta(G_0)}(M)\cap \Lie(K)=\mathcal{V}_{U_\zeta(K_0)}(M)$. Thus $\mathcal{V}_{U_\zeta(H_0)}(M)=0$ and hence $M$ is a projective $U_\zeta(H_0)$-module.
\end{proof}

\begin{thm}
Let $\mathfrak{g}\neq \mathfrak{sl}_2$ be a simple Lie algebra. Let $B$ be a block of $U_\zeta(G_r)$. Then there are no Euclidean components for $B$. If $r=0$, then all components of $U_\zeta(G_0)$ containing $U_\zeta(G)$-modules are isomorphic to $\mathbb{Z}[\mathbb{A}_\infty]$, in particular those containing simple modules.
\end{thm}

\begin{proof}
First suppose that $r=0$. By Proposition \ref{cx2} we have that each Euclidean component has complexity two. But then the foregoing theorem implies $\mathfrak{g}=\mathfrak{sl}_2$ as every Euclidean component contains a simple module and all simple modules are modules for $U_\zeta(\mathfrak{g})$.\\
Now suppose that $r\geq 1$. Let $\Theta$ be a component isomorphic to $\mathbb{Z}[\Delta]$, where $\Delta$ is Euclidean. The component $\Omega^{-1}\Theta\cong \Theta$ is not regular, i.e. is attached to a projective module $P$ (see e.g. \cite[Main Theorem]{KZ11}). Therefore $\Theta$ contains the simple module $\Omega P/\soc(P)$. This simple module is of the form $\overline{S}\otimes S$, where $\overline{S}$ is a simple module for $\Dist(G_r)$, regarded as a module for $U_\zeta(G_r)$ and $S$ is a simple $U_\zeta(G_0)$-module (cf. \cite[Lemma 1.1]{K11}). Then as a module for $U_\zeta(G_r)$ it has complexity two, hence as a module for $U_\zeta(G_0)$, $S$ has complexity less or equal two. By the foregoing proposition this is not possible for $\mathfrak{g}\neq \mathfrak{sl}_2$.
\end{proof}

\section{The component $\mathbb{Z}[\mathbb{A}_\infty]$}

We have seen in the previous section that the simple modules are (at least for $r=0$) in components of type $\mathbb{Z}[\mathbb{A}_\infty]$. Therefore it is worth studying this case in more detail. The approach is based on the classical case as considered in \cite{FR11}.\\

Recall that a module $M$ is said to be quasi-simple if the middle term of the Auslander-Reiten sequence ending in $M$ is indecomposable.

\begin{thm}
Let $\mathfrak{g}$ be a simple Lie algebra. If a simple module for $U_\zeta(G_r)$ belongs to a component of type $\mathbb{Z}[\mathbb{A}_\infty]$, then it is quasi-simple.
\end{thm}

\begin{proof}
Let $S$ be a simple module contained in some component $\Theta\cong \mathbb{Z}[\mathbb{A}_\infty]$. Suppose $S$ is not quasi-simple. Then since $U_\zeta(G_r)$ is symmetric, by \cite[Theorem 1.5]{Kaw97} there exist simple modules $T_0, T_1,\dots T_n$, such that $T_0\cong S$, the projective cover of $T_i$ is uniserial of length $n+2$ with $\Kopf \rad P(T_i)\cong T_{i-1}$. By the same token $T_i$ has multiplicity one in $P(T_i)$ for all $i$.\\
If $n\geq 2$, then this yields the contradiction $\dim \Ext^1_{U_\zeta(G_r)}(T_1,T_2)=0$ while $\dim \Ext^1_{U_\zeta(G_r)}(T_2,T_1)=1$ (cf. \cite[3.8]{AM11}). Therefore $n=1$, i.e. $P(T_1)$ has length three. By \cite[Proposition B.9]{DruPhD}, the projective cover of $S$ has a filtration by baby Verma modules. Hence one baby Verma module has to be simple or projective. But by the same token this implies that it is simple and projective. Hence $S$ is the Steinberg module for $U_\zeta(G_r)$, a contradiction.
\end{proof}

\begin{cor}
Let $\mathfrak{g}$ be a simple Lie algebra. Let $S$ be a simple non-projective $U_\zeta(G_r)$-module in a component of type $\mathbb{Z}[\mathbb{A}_\infty]$. Then $\heart(P(S))$ is indecomposable.
\end{cor}

\begin{proof}
Since $U_\zeta(G_r)$ is symmetric, we have $\soc P(S)\cong S$. Hence the following sequence is the standard almost split sequence originating in $\rad{P(S)}$:
\[0\to \rad{P(S)}\to P(S)\oplus \heart P(S)\to P(S)/S\to 0.\]
The autoequivalence on the stable module category $\Omega$ induces an automorphism on the stable Auslander-Reiten quiver. Hence $S$ and $\Omega^{-1}S$ have the same number of non-projective predecessors. By the foregoing theorem we thus have that $\heart P(S)$ is indecomposable.
\end{proof}

\begin{thm}
Let $\mathfrak{g}$ be a simple Lie algebra. Let $\ell>1$ be an odd integer not divisible by $3$ if $\Phi$ is of type $\mathbb{G}_2$. Furthermore for $\charac k=0$ assume that $\ell$ is good for $\Phi$ and that $\ell>3$ if $\Phi$ is of type $\mathbb{B}_n$ or $\mathbb{C}_n$. For $\charac k=p>0$ assume that $p$ is good for $\Phi$ and that $\ell>h$. Then there is only one simple $U_\zeta(G_r)$-module in each Auslander-Reiten component of type $\mathbb{Z}[\mathbb{A}_\infty]$.
\end{thm}

\begin{proof}
We first prove the case of $r=0$ and then combine this result with the classical case to get the result for all $r$. If $r=0$, suppose there is another simple module $T$ in $\Theta$. Then $S$ and $T$ are both quasi-simple by the foregoing considerations. Hence they lie in the same $\tau$-orbit, without loss of generality let $\Omega^{2n}_{U_\zeta(G_0)}T\cong S$. Accordingly we have $\dim \Ext^{2n}_{U_\zeta(G_0)}(T,T)=\dim \Hom_{U_\zeta(G_0)}(S,T)=0$. Hence $T$ is projective by Proposition \ref{generalsupportfacts} and Theorem \ref{upperbound}. If $r\neq 0$, then by the Steinberg tensor product theorem (see e.g. \cite[Lemma 1.1]{K11}), simple modules are tensor products of a simple module for $U_\zeta(G_0)$ and a simple module for $\Dist(G_r)$. Let $\overline{T}\otimes T$ and $\overline{S}\otimes S$ be such simple modules. Then $\Omega^{2n}_{U_\zeta(G_r)}\overline{S}\otimes S\cong \overline{T}\otimes T$ implies that $T^{\dim \overline{T}}\cong \Omega^{2n}_{U_\zeta(G_0)}S^{\dim \overline{S}}\oplus P$, where $P$ is projective. In particular, since $\Omega^{2n}_{U_\zeta(G_0)}S$ is indecomposable we have $\Omega^{2n}_{U_\zeta(G_0)}S\cong T$. Hence $S\cong T$ is a projective or a periodic module for $U_\zeta(G_0)$, hence it is the Steinberg module for $U_\zeta(G_0)$. But then the result follows from the classical case considered in \cite[Theorem 2.6]{FR11}, because these modules form a block ideal equivalent to $\modu \Dist(G_r)$ (cf. \cite[Proposition 1.3]{K11}).
\end{proof}

\section*{Acknowledgement}
The results of this article are part of my doctoral thesis, which I am currently writing at the University of Kiel. I would like to thank my advisor Rolf Farnsteiner for his continuous support, especially for answering my questions on the classical case. I also would like to thank Chris Drupieski for answering some of my questions on Frobenius-Lusztig kernels, especially on restrictions of the parameter. Furthermore I thank the members of my working group for proof reading.

\bibliographystyle{alpha}
\bibliography{publication}

\end{document}